\documentclass[a4paper, 12pt, reqno]{amsart}
\usepackage{amssymb}
\usepackage{mathrsfs}
\usepackage{dsfont}
\usepackage{stmaryrd, MnSymbol}
\usepackage{enumerate, xspace}
\usepackage{changebar}

\newtheorem{lemma}{Lemma}[section]
\newtheorem{corollary}{Corollary}[section]
\newtheorem{proposition}{Proposition}[section]

\newtheorem{example}{Example}[section]

\numberwithin{equation}{section}

\def\R{{\mathbb{R}}}
\def\N{{\mathbb{N}}}
\def\Z{{\mathbb{Z}}}
\def\T{{\mathbb{T}}}

\begin{document}

\title{An elementary and direct computation of cohomology with and without a group action}
\author{Makiko Sasada \\ Graduate School of Mathematical Sciences, \\ The University of Tokyo}
\address{
Makiko Sasada\\
Graduate School of Mathematical Sciences, The University of Tokyo\\
3-8-1, Komaba, Meguro-ku, Tokyo, 153-8914, Japan.}
\email{\tt sasada@ms.u-tokyo.ac.jp}


\maketitle

\begin{abstract}
Recently, we introduced a configuration space with interaction structure and a uniform local cohomology on it with co-authors in \cite{BKS}. The notion is used to 
understand a common structure of infinite product spaces appeared in the proof of Varadhan's non-gradient method. For this, the cohomology of the configuration space with 
a group action is the main target to study, but the cohomology is easily obtained from that of the configuration space without a group action by applying 
a well-known property on the group cohomology. In fact, the analysis of the cohomology of the configuration space without a group action is the essential part of \cite{BKS}. In this article, we give an 
elementary and direct proof to obtain the cohomology of a space with a group action from that without a group action under a certain condition including the setting of 
the configuration space with interaction structure. In particular, no knowledge of group cohomology is required.
\end{abstract}



\maketitle

\section{Group action on linear spaces, homomorphism and kernel}

Suppose $U$ and $W$ are $\R$-linear spaces and $\pi :U \to W$ is an $\R$-linear map. We also assume that a group $G$ acting on $U$ and $W$, namely, for each $g \in G$,
\[
g: U \to U, \quad g : W \to W
\]
are automorphisms and $(gh)(u)=g(h(u))$ for any $g,h \in G$ and $u \in U$ or $W$. Moreover, we assume 
\[
g \circ \pi = \pi \circ g
\]
for all $g \in G$. Let $U^G$ and $W^G$ are the linear subspaces of $U$ and $W$ which are invariant under the action of $G$: 
\[
U^G:=\{u \in U \ ; \ g(u)=u, \forall g \in G\}, \quad W^G:=\{w \in W \ ; \ g(w)=w, \forall g \in G\}.
\]
We also denote by $\pi(U)^G := \pi(U) \cap W^G$, which is the intersection of the image of $U$ and the invariant set under the group action. Since $g \circ \pi = \pi \circ g$, it is obvious that $\pi(U^G) \subset \pi(U)^G$.

\begin{proposition}\label{prop:decomp}
Assume that $m:= \dim \ker \pi \in \N$ and $G$ is generated by a finite subset $\{g_1,\dots, g_d\}$. Then, 
\[
\dim \pi(U)^G/ \pi(U^G) \le md.
\]
Moreover, $\dim \pi(U)^G/ \pi(U^G) = md$ if and only if the following two conditions both hold : (i) $\ker \pi \subset  U^G$ and (ii) $(\ker \pi)^d \subset \bar{g} (U)$ where $\bar{g} : U \to U^d$ is the linear map $(\bar{g}u)_i =g_i u -u=(g_i-\mathrm{id})u$. 
\end{proposition}
\begin{proof}
Since $\dim \ker \pi =m$, $\ell:=\dim(\bar{g}(U) \cap  (\ker \pi)^d ) \le md$. Let $\mathbf{u}^{(1)},\dots, \mathbf{u}^{(\ell)}$ be a basis of the finite dimensional linear space $\bar{g}(U) \cap  (\ker \pi)^d$. Then, by definition, for each $k \in \{1,2,\dots, \ell\}$, there exist $u^{(k)} \in U$ such that $\bar{g} u^{(k)}=\mathbf{u}^{(k)}$ or equivalently $\mathbf{u}^{(k)}_i=g_i u^{(k)} - u^{(k)} $ for $i=1,\dots,d$. 
Note that since $\mathbf{u}^{(k)} \neq \mathbf{0}$, $u^{(k)} \notin U^G$ and since $\{ \mathbf{u}^{(1)},\dots, \mathbf{u}^{(\ell)} \}$ is linearly independent, $\{ u^{(1)},\dots,u^{(\ell)} \}$ is also linearly independent. 
Now, suppose $w \in \pi(U)^G$. Then, $w=\pi(u)$ for some $u \in U$ and $w \in W^G$. Hence 
\[
\pi( (g_i-\mathrm{id})u)=\pi(g_i u)-\pi(u)=g_i(\pi(u))-\pi(u)=g_i(w)-w=0,
\] 
and so $(g_i-\mathrm{id})u \in \ker \pi$ for $i=1,\dots,d$. Hence, there exists $a_1,\dots,a_{\ell}$ such that $\bar{g} u=\sum_{k=1}^{\ell}a_k\mathbf{u}^{(k)}$. Let $\bar{u}:=u-\sum_{k=1}^{\ell}a_ku^{(k)}$. Then, 
\[
(g_i-\mathrm{id})\bar{u}=(g_i-\mathrm{id})u-\sum_{k=1}^{\ell}a_k\mathbf{u}^{(k)}_i=(\bar{g}u)_i - (\sum_{k=1}^{\ell}a_k\mathbf{u}^{(k)})_i =0
\]
for $i=1,\dots,d$ and so $\bar{u} \in U^G$. Therefore, we have
\[
w=\pi(u)=\pi(\sum_{k=1}^{\ell}a_ku^{(k)} + \bar{u})=\sum_{k=1}^{\ell}a_k \pi(u^{(k)})+\pi(\bar{u})
\]
where $\pi(\bar{u}) \in \pi(U^G)$ and conclude that 
\[
\dim \pi(U)^G/ \pi(U^G) \le \ell \le dm.
\]
From the above argument, $\dim \pi(U)^G/ \pi(U^G) =dm$ if and only if the following two conditions both hold : 
\begin{enumerate}
\item $\ell=dm$, 
\item $\sum_{k=1}^{\ell} a_k\pi(u^{(k)}) \in \pi(U^G)$ implies $a_1=a_2=\dots=a_{\ell}=0$. 
\end{enumerate} 
It is simple to see that the condition (1) is equivalent to $\bar{g} (U) \cap  (\ker \pi)^d = (\ker \pi)^d$ and also to $(\ker \pi)^d \subset \bar{g} (U)$.
Next, we prove that the condition (2) is equivalent to $\ker \pi \subset U^G$. For this, first note that $\sum_{k=1}^{\ell} a_k\pi(u^{(k)}) \in \pi(U^G)$ is equivalent to the existence of $\bar{u} \in U^G$ such that
\[
\sum_{k=1}^{\ell} a_k u^{(k)} -\bar{u} \in \ker \pi.
\]
Hence, if $\ker \pi \subset U^G$, then $\sum_{k=1}^{\ell} a_k u^{(k)} \in U^G$ and so applying $\bar{g}$ to $\sum_{k=1}^{\ell} a_k u^{(k)}$, we obtain $a_1=a_2=\dots=a_{\ell}=0$. On the other hand, if $\ker \pi \nsubset U^G$,
then there exists $u^* \in \ker \pi$ such that $\bar{g} u^* \neq 0$. Since $\bar{g}u^* \in \bar{g}(U) \cap (\ker \pi)^d$, there exists $a_1,a_2,\dots ,a_{\ell}$ such that $\bar{g}u^*=\sum_{k=1}^{\ell}a_k \mathbf{u}^{(k)}$ satisfying 
$a_k \neq 0$ for some $k$. This implies $u^* - \sum_{k=1}^{\ell} a_k u^{(k)} \in U^G=\ker \bar{g}$ and so $\pi(u^*) -  \sum_{k=1}^{\ell} a_k \pi (u^{(k)} ) \in \pi(U^G)$. Since $\pi(u^*)=0$, this means $\sum_{k=1}^{\ell} a_k \pi (u^{(k)} ) \in \pi(U^G)$ and so the condition (2) does not hold.
\end{proof}

From this proposition, the following useful corollary directly follows.

\begin{corollary}\label{cor:decomp}
Suppose that $\{u_1,\dots,u_m\}$ is a basis of $\ker \pi$ and $G=<g_1,g_2,\dots,g_d>$. Moreover, assume that\\
(i) $\ker \pi \subset U^G$. \\
(ii) There exist $u_{j,k} \in U$ such that $(g_i-\mathrm{id})u_{j,k}=\delta_{i,j}u_k$ for $i,j=1,\dots,d$ and $k=1,\dots,m$.\\
Then, for any $w \in \pi(U)^G$, there exists $u \in U^G$ and a sequence of real numbers $(a_{j,k})_{j=1,\dots,d, k=1\dots,m}$ which does not depend on the choice of $u_{j,k}$ but depends on the basis $\{u_1,\dots,u_m\}$ and 
the generating set $\{g_1,g_2,\dots, g_d \}$ such that
\[
w=\pi(\sum_{j=1}^d\sum_{k=1}^m a_{j,k}u_{j,k} +u).
\]
\end{corollary}
\begin{proof}
The second condition clearly implies that $\mathbf{u}^{(j,k)}:=\bar{g} u_{j,k}  \in U^d$, $j=1,\dots, d$, $k=1,\dots,m$ is a basis of $(\ker \pi)^d$ and $(\ker \pi)^d \subset \bar{g} (U)$. Hence, from the last proposition and its proof, the existence of $u$ and $(a_{j,k})_{j=1,\dots,d, k=1\dots,m}$ follows. Next, we prove that $a_{j,k}$ is independent from the choice of $u_{j,k}$. Suppose that $\{\tilde{u}_{j,k}\}$ also satisfies
\[
(g_i-\mathrm{id})\tilde{u}_{j,k}=\delta_{ij}u_k,
\]
and $w=\pi(\sum_{j=1}^d\sum_{k=1}^m a_{j,k}u_{j,k} +u)= \pi(\sum_{j=1}^d\sum_{k=1}^m \tilde{a}_{j,k}\tilde{u}_{j,k} +\tilde{u})$ for some $\tilde{a}_{j,k} \in \R$ and $\tilde{u} \in U^G$. Then, since $(g_i-\mathrm{id})(u_{j,k}-\tilde{u}_{j,k})=0$ for all $i$, $u_{j,k}-\tilde{u}_{j,k} \in U^G$ for all $j,k$. Hence, for some $\tilde{u}' \in U^G$,
\[
w=\pi(\sum_{j=1}^d\sum_{k=1}^m a_{j,k}u_{j,k} +u)= \pi(\sum_{j=1}^d\sum_{k=1}^m \tilde{a}_{j,k}{u}_{j,k} +\tilde{u}').
\]
In particular, $\sum_{j=1}^d\sum_{k=1}^m (a_{j,k}-\tilde{a}_{j,k})\pi(u_{j,k})  \in \pi(U^G)$. Hence, $a_{j,k}=\tilde{a}_{j,k}$ for all $j,k$ by the proof of the last proposition.
\end{proof}

\subsection{Under the additional condition : $\ker \pi \subset U^G$} 

In this subsection, we always assume that $\ker \pi \subset U^G$. Let $\tilde{U} := \{ u \in U ; \pi(u) \in W^G\}$ and $\tilde{W}:=W^G$, and $\tilde{G} := G^{ab}$ where $G^{ab}$ is the abelianization of $G$. 
First, we show the following simple but important lemma. 
\begin{lemma}
Suppose $\ker \pi \subset U^G$. If $[g] =[h] \in G^{ab}$, then $g u =h u$ for any $u \in \tilde{U}$ where $[g], [h] \in G^{ab}$ are the representative of $g, h \in G$ respectively.
\end{lemma} 
\begin{proof}
It is sufficient to show that $g h u = h g u$ for any $g, h \in G$ and $u \in \tilde{U}$. For any $u \in \tilde{U}$, since $\pi(u) \in W^G$, $gu- u \in \ker \pi$ and $hu- u \in \ker \pi$. Then, since $\ker \pi \subset U^G$, 
$h(gu-u)= gu-u$ and $g(hu-u)=hu-u$ and so $h gu=hu+gu-u= ghu$. 
\end{proof}

Applying the above lemma, $\tilde{G}$ acts on $\tilde{U}$ and $\tilde{W}$ as follows :
\[
[g ] u =g u, \quad [g] w =w 
\]
for $u \in \tilde{U}, w \in \tilde{W}$ and $[g] \in G^{ab}$ is the representative of $g \in G$. Moreover, $\pi : \tilde{U} \to \tilde{W}$ also defines a linear map and $[g] \circ \pi = \pi \circ [g]$ holds. Since $\pi(U)^G=\pi(\tilde{U})^G=\pi(\tilde{U})^{\tilde{G}}$ and $\pi(U^G)=
\pi(\tilde{U}^G)= \pi (\tilde{U}^{\tilde{G}})$, we also have $\pi(U)^G/ \pi(U^G) = \pi(\tilde{U})^{\tilde{G}} /\pi (\tilde{U}^{\tilde{G}})$.

\begin{proposition}\label{prop:general}
Assume that $\ker \pi \subset U^G$, $m:= \dim \ker \pi \in \N$ and $d:=\mathrm{rank} \ G^{ab}$. Then, 
\[
\dim \pi(U)^G/ \pi(U^G) \le md.
\]
Moreover, for a basis $\{u_1,u_2,\dots,u_m\}$ of $\ker \pi$ and a generating set $\{g_1,g_2,\dots, g_d\}$ of the free part of $G^{ab}$, if there exist $u_{j,k} \in \tilde{U}$ such that $(g_i-\mathrm{id})u_{j,k}=\delta_{i,j}u_k$ for $i,j=1,\dots,d$ and $k=1,\dots,m$, then $\dim \pi(U)^G/ \pi(U^G) = md$ holds. In addition, for any $w \in \pi(U)^G$, there exists $u \in U^G$ and a sequence of real numbers $(a_{j,k})_{j=1,\dots,d, k=1\dots,m}$ which does not depend on the choice of $u_{j,k}$ but depends on the basis $\{u_1,\dots,u_m\}$ and 
the generating set $\{g_1,g_2,\dots, g_d \}$ such that
\[
w=\pi(\sum_{j=1}^d\sum_{k=1}^m a_{j,k}u_{j,k} +u).
\]
\end{proposition}
\begin{proof}
Since $\pi(U)^G/ \pi(U^G) = \pi(\tilde{U})^{\tilde{G}} /\pi (\tilde{U}^{\tilde{G}})$, we only work on $\tilde{U}, \tilde{W}$ and $\tilde{G}$.  Let $g \in \tilde{G}$ be a torsion and $n$ be its order. Then, for any $u \in \tilde{U}$, since $gu -u \in \ker \pi \subset  U^G$, 
\[
0= (\overbrace{ g \circ g \circ \dots \circ g}^{n}  -\mathrm{id}) u= \sum_{q=0}^{n-1} \overbrace{ g \circ g \circ \dots \circ g}^{q} (g u- u) = n (g  u -u).
\]
Hence, $gu =u$. Then, by letting $G^*$ be the free part of $\tilde{G}$, we have $\pi(\tilde{U})^{\tilde{G}} = \pi(\tilde{U})=\pi(\tilde{U})^{G^*}$ and $\pi (\tilde{U}^{\tilde{G}}) = \pi (\tilde{U}^{G^*})$. Now, we can apply Proposition \ref{prop:decomp} and Corollary 
\ref{cor:decomp} to $\tilde{U}, \tilde{W},\pi$ and $G^*$.
\end{proof}

\section{Application}
In this section, we consider the application of the results in the last section for some concrete settings. 

The first example is one of the most typical settings where the linear space of $U$ and $V$ are differential forms of a manifold. 
\begin{example}
Let $U$ and $W$ be spaces of smooth $0$ and $1$-differential forms on $\R^d$, $U=C^0(\R^d),W=C^1(\R^d)$ and $\pi :U \to W$ be the differential operator 
\[
\pi : U \to W, \quad \pi(f)=\sum_{i=1}^d\frac{\partial f}{\partial x_i}dx_i. 
\]
Then, $\ker\pi$ is the set of constant functions and in particular $\ker \pi \cong \R$. Let $G=\Z^d$ acting on $\R^d$ as
\[
g(x)=x+g \quad x \in \R^d, g \in G.
\]
The action of $G$ to $U$ and $W$ are also naturally induced and satisfies $g \circ \pi= \pi \circ g$ for any $g \in G$. For this setting, it is easy to see that
\[
\ker \pi \subset U^G.
\]
Denote a constant function on $\R^d$ by $\mathbf{1}$ and $(e_i)_{i=1,\dots, d}$ be the normal basis of $G=\Z^d$. Then, 
\[
(e_i-\mathrm{id})f_{j}=\delta_{ij}\mathbf{1}
\]
where $f_j : \R^d \to \R$ is the $j$-th coordinate function $f_j(x)=x_j$. Hence, by Proposition \ref{prop:decomp}, $\dim \pi(U)^G/ \pi(U^G) = d$. Note that
since $H^1(\R^d) \cong \{0\}$, $w \in \pi(U)$ is equivalent to $dw=0$, namely $w \in C^1(\R^d)$ is exact if and only if it is closed. Hence, we have
\[
\pi(U)^G \cong \{ w \in C^1(\T^d) ; dw= 0 \} 
\]
and 
\[
\pi(U^G) \cong \{ w \in C^1(\T^d) ; w= df , f  \in C^0(\T^d)\}. 
\]
In particular,
\[
\pi(U)^G/ \pi(U^G) \cong H^1(\R^d/G) = H^1 (\T^d) \cong \R^d.
\]
From Corollary \ref{cor:decomp}, we have that if $w \in C^1(\R^d)$ is closed (namely $dw=0$) and invariant under the action of $G$, then there exists $a_1,\dots,a_d$ and $f \in C^0(\R^d)$ which is invariant under the actin of $G$ such that
\[
w=\pi(\sum_{i=1}^da_if_i+f).
\]
\end{example}

In the next example, we apply the main result to the discrete geometry. The computation is also important for the application to the configuration space.

\begin{example}\label{exm:graph}
Let $X=(V,E)$ be a symmetric directed graph and
\[
C^0(X):=\{f : V \to \R\}, \quad C^1(X):=\{w: E\to \R; w(e)=-w(\bar{e})\}.
\]
Define $U=C^0(X)$ and $W=C^1(X)$ and let $\pi : U \to W$ be
\[
\pi(f)(e)=f(te)-f(oe).
\]
For this setting $\dim \ker \pi$ is equal to the number of connected components of the graph $X$. Suppose $X$ has $m$ connected components and also a group $G$ acting on $X$, namely for each $g \in G$, there is a graph automorphism of $X$, which we also denote by $g$, and $g( h(v))=(gh)(v)$ for any $g,h \in G$. 
\begin{proposition}\label{prop:graph}
Suppose $G \cong \Z^d$. Then, $\dim \pi(U)^G / \pi(U^G) \le dm$. Moreover, $\dim \pi(U)^G / \pi(U^G) = dm$ holds if and only if the action of $G$ to the graph $X$ is free and closed in each connected component, namely for any $v \in V$ and $g \in G$, $v$ and $g(v)$ are in a same connected component of $X$.
\end{proposition}
\begin{proof}
By Proposition \ref{prop:decomp}, we have $\dim \pi(U)^G / \pi(U^G) \le dm$. Next, we study when $\dim \pi(U)^G / \pi(U^G) = dm$ holds. Denote the connected components of $X$ by $\{X_k=(V_k,E_k)\}_{k=1,\dots,m}$, namely $V=\bigsqcup  V_k$, $E=\bigsqcup  E_k$ and for each $k$, $(V_k,E_k)$ is a connected graph. Then, $\{ \mathbf{1}_{V_k}\}_{k=1,\dots,m}$ is a basis of $\ker \pi$. First, suppose $\dim \pi(U)^G / \pi(U^G) = dm$. Then, from Proposition \ref{prop:decomp}, $ \mathbf{1}_{V_k} \in U^G$ for all $k$, namely the action of $G$ is closed in each connected component. Moreover, there exists a set of functions $f_{j,k} \in C^0(V)$ such that 
\[
(g_i-\mathrm{id})f_{j,k}=\delta_{ij}\mathbf{1}_{V_k}
\]
where $\{g_i\}_{i=1,\dots,d}$ is a generator of $G$. We prove that under this condition, the action of $G$ must be free. In fact, if it is not free, then there exists $k$, $v \in V_k$ and $g \in G$ such that $g=\sum_{i=1}^d b_i g_i$ where $b_1 \neq 0$ without loss of generality and $g(v)=v$. Then, 
\[
f_{1,k}(g(v))-f_{1,k}(v)=\sum_{\ell=1}^d \big(f_{1,k}((\sum_{i=1}^{\ell}b_ig_i)(v))-f_{1,k}((\sum_{i=1}^{\ell-1}b_ig_i)(v))\big)
\]
and 
\begin{align*}
& f_{1,k}((\sum_{i=1}^{\ell}b_ig_i)(v))-f_{1,k}((\sum_{i=1}^{\ell-1}b_ig_i)(v))=\sum_{p=0}^{b_{\ell}-1} (g_{\ell}-\mathrm{id})f_{1,k} ((\sum_{i=1}^{\ell-1}b_ig_i + pg_{\ell})(v)) \\
& =\delta_{\ell,1}  \sum_{p=0}^{b_{\ell}-1} \mathbf{1}_{V_k}((\sum_{i=1}^{\ell-1}b_ig_i + pg_{\ell})(v)) =\delta_{\ell,1}b_{\ell}
\end{align*}
where in the last equation we use that $g(v) \in V_k$ for any $g \in G$ since $ \mathbf{1}_{V_k} \in U^G$. So, we conclude that 
\[
f_{1,k}(g(v))-f_{1,k}(v)= b_1 \neq 0
\]
which contradicts with $g(v)=v$. Hence, the action of $G$ must be free if $\dim \pi(U)^G / \pi(U^G) = dm$. \\
Next, we prove the opposite direction. Suppose the action of $G$ is free and $ \mathbf{1}_{V_k} \in U^G$ for all $k$. Denote $V_0:=V/G$ and for each $[v_0] \in V_0$, fix a representative element $v_0 \in V$. We also denote the set of these representative elements by $V_0$ by an abuse of notation where $V_0 \subset V$. Then, for each $v \in V$, there exists unique $v_0 \in V_0$ and $(b_1,b_2,\dots b_d) \in \Z^d$ such that $v=(\sum_{i=1}^db_ig_i)(v_0)$. Note that the uniqueness follows from the freeness of the action. Now, let $f_{j,k}(v):=b_j\mathbf{1}_{V_k}(v)=b_j\mathbf{1}_{V_k}(v_0)$ for $v=(\sum_{i=1}^db_ig_i)(v_0)$. It is easy to see that these functions satisfy
\[
(g_i-\mathrm{id})f_{j,k}=\delta_{ij}\mathbf{1}_{V_k}.
\]
Hence, applying Proposition \ref{prop:decomp}, if the action of $G$ is free and closed in each connected component, then $\dim \pi(U)^G / \pi(U^G) = dm$.
\end{proof}
From this proposition and Corollary \ref{cor:decomp}, we conclude that if the action of $G$ is free and closed in each connected component, then for any $w \in C^1(X)$, if $w$ is closed and invariant under the action of $G$, there exists $\{a_{i,k}\}_{i=1,\dots,d, k =1,\dots,m}$ and $f \in C^0(X)$ which is invariant under the action of $G$ such that
\[
w=\pi(\sum_{i=1}^d\sum_{k=1}^m a_{i,k}f_{i,k}+f)
\]
where $f_{i,k}$ is the function constructed in the proof of Proposition \ref{prop:graph}.

For more general group $G$, we have the following result. 
\begin{proposition}
Suppose the graph $X$ is connected and $d$ be the rank of $G^{ab}$. Then,  $\dim \pi(U)^G / \pi(U^G) \le d$. Moreover, $\dim \pi(U)^G / \pi(U^G) = d$ holds if and only if the action of $G$ to the graph $X$ is free.
\end{proposition}
\begin{proof}
For this case, $m=1$ and the kernel of $\pi$ is the set of constant functions. Hence $\ker \pi \subset U^G$ holds automatically and so by Proposition \ref{prop:general}, $\dim \pi(U)^G / \pi(U^G) \le d$. Let $\{g_1,g_2,\dots,g_d\}$ be the generating set of the free part of $G^{ab}$. If $\dim \pi(U)^G / \pi(U^G) = d$, then $\dim \pi(\tilde{U})^{G^*} / \pi(\tilde{U}^{G^*}) = d$ and so there exists 
\[
(g_i-\mathrm{id})f_{j}=\delta_{ij}\mathbf{1}_{V}
\]
where $f_j \in \tilde{U}$. Hence, by the same argument as the last proposition, we conclude that the action of $G$ must be free. The opposite direction is also shown by a similar way. 
\end{proof}
\end{example}

Finally, we give an application to the configuration space with interaction structure, though we do not give a precise definition.
\begin{example}
Let $(S^X,\Phi)$ be a configuration space with transition structure associated to a triplet $(X,S,\phi)$ where $X$ is a locale, $S$ is a set of states and $\phi$ is an interaction on $S$ (see Section 2, \cite{BKS}). Define $U=C^0_{\mathrm{unif}}(S^X)$ and $W=Z^1_{\mathrm{unif}}(S^X)$ and let $\pi : U \to W$ be the deferential $\partial$ defined by the transition structure (see Section 3, \cite{BKS}). Then, Theorem 6 of \cite{BKS} implies  
that under the assumption of the theorem, $\pi (U)=W$ and $\dim \ker \pi = \dim \mathrm{Consv}^{\phi}(S)$. Moreover, from Theorem 3.7 of \cite{BKS}, the form of the kernel of $\pi$ is explicitly known, namely 
\[
\ker \pi =\{ \xi_X =\sum_{x \in X} \xi_x  ; \xi \in  \mathrm{Consv}^{\phi}(S) \}. 
\]
Now, we consider the case with a group action and apply our main result to deduce Theorem 5 of \cite{BKS} from Theorem 6 of \cite{BKS}. First, note that for this setting, the action of a group $G$ on $U$ and $V$ are induced from the action of $G$ on the locale, namely the underlying graph $X$. Because of this structure and the explicit form os the kernel of $\pi$, it is easy to see that $\ker \pi \subset U^G$. In fact, $g \xi_ X = \sum_{x \in X} \xi_{gx} = \sum_{x \in X} \xi_{x}$ since $g : X \to X$ is a bijection. Hence, we can apply Proposition \ref{prop:general}. Since the locale $X$ is assumed to be connected, if the action of $G$ on $X$ is free and the rank of $G^{ab}$ is $d$, choosing a generator $\{ g_1,g_2,\dots,g_d\}$, by the last example, there exists $f_j : X \to \R$ such that 
\[
(g_i-\mathrm{id})f_{j}=\delta_{ij}\mathbf{1}_{X}
\]
for $i,j =1,2,\dots, d$. Then, by the direct computation, for each $\xi \in \mathrm{Consv}^{\phi}(S)$, $\tilde{f}_j :=\sum_{x \in X} f_j(x) \xi_x \in U$ satisfies
\begin{align*}
& (g_i-\mathrm{id})\tilde{f}_{j}=\sum_{x \in X} f_j(x) \xi_{ g_i x} - \sum_{x \in X} f_j(x) \xi_{x} = \sum_{x \in X} (f_j(x) -f_j(g_i x)) \xi_{g_ix} \\
& = - \sum_{x \in X} \delta_{ij}\xi_{g_ix} =-\delta_{ij} \xi_X.
\end{align*}
Finally, since $\pi(U)^G =W^G= \mathcal{C}$ and $\pi(U^G)= \mathcal{E}$, we obtain Theorem 5 of \cite{BKS}.
\end{example}

\section*{Acknowledgement}

The author expresses her sincere thanks to Professor Kenichi Bannai for their insightful discussions.

\end{document}